\documentclass[11pt]{article}
\usepackage[utf8]{inputenc}
\usepackage[T1]{fontenc}
\usepackage{amsmath,amssymb,amsthm}
\usepackage{enumitem}
\usepackage{cite}

\newcommand{\sr}{\hat r}
\newcommand{\sri}{\hat r_{\infty}}
\newcommand{\core}{\operatorname{core}}

\newtheorem{theorem}{Theorem}[section]
\newtheorem{proposition}[theorem]{Proposition}
\newtheorem{lemma}[theorem]{Lemma}
\newtheorem{corollary}[theorem]{Corollary}

\theoremstyle{definition}
\newtheorem{problem}[theorem]{Problem}

\theoremstyle{remark}
\newtheorem{remark}[theorem]{Remark}

\newcounter{case}
\newenvironment{case}{\refstepcounter{case}\par\medskip\noindent\textbf{Case \arabic{case}. }}{\par\medskip}

\begin{document}

\title{\textbf{The Erd\H{o}s--Faudree Problems and the Isolate-Free Core\footnote{Supported by the National Science Foundation of China
        (No.~12471329 and 12061059).}}}

\author{Yaping Mao\footnote{Academy of Plateau
Science and Sustainability, and School of
Mathematics and Statistis, Qinghai Normal University, Xining,
Qinghai 810008, China. {\tt yapingmao@outlook.com; myp@qhnu.edu.cn}}}

\date{}
\maketitle

\begin{abstract}
In 1981, Erd\H{o}s and Faudree asked whether there exists an infinite family of graphs
$G_N$ on $N$ vertices with $\Delta(G_N)<N-1$ and $\sri(G_N)=1$, and whether every
family with $|V(G_N)|=N$ and $\Delta(G_N)<c$ for some fixed constant $c$ must satisfy
$\sri(G_N)\to 0$. We show first that the literal forms of the two questions are controlled
entirely by isolated vertices: for every nonempty graph $G$, the whole sequence
$\bigl(\sr(tK_2,G)\bigr)_{t\ge 1}$ depends only on the isolate-free core $\core(G)$.
Consequently, Problem~1 has a positive answer and Problem~2 has a negative answer in
exactly their original form.
We then turn to the genuine content behind the two problems. For Problem~1 we study
connected graphs and prove a complete limit theorem: for every $\alpha\in[0,1]$ there
exists a family of connected bipartite graphs $G_N$ with $|V(G_N)|=N$ and
$\sri(G_N)\to\alpha$; in particular there are connected graphs with $\Delta(G_N)=N-2$ and
$\sri(G_N)\to 1$. For Problem~2 we prove a strengthened positive statement: if
$\Delta(G_N)<c$ for a fixed constant $c$ and the isolate-free core of $G_N$ has order tending
to infinity, then $\sri(G_N)\to 0$. In particular every connected bounded-degree family
satisfies $\sri(G_N)\to 0$. Thus the original Erd\H{o}s--Faudree questions are resolved in
their literal form, and the mechanism behind their connected and disconnected behavior is
identified precisely.\\[2mm]
\end{abstract}


\section{Introduction}

All graphs in this paper are finite, simple, and undirected. A graph is
\emph{nonempty} if it has at least one edge. For graphs $F,G,H$ we write
$F\to(G,H)$
to mean that every red--blue coloring of $E(F)$ contains either a red copy of $G$
or a blue copy of $H$. The \emph{size Ramsey number} of $(G,H)$ is
\[
\sr(G,H):=\min\{|E(F)|:F\to(G,H)\}.
\]
The notion was introduced by Erd\H{o}s, Faudree, Rousseau, and Schelp
\cite{EFRS78}. For more details, we refer to a survey paper
\cite{FaSh02} by Faudree and Schelp and some papers \cite{Beck83,
Bielak87, CJK18, DJKR21,EFRS78,ErdosFaudree84, FaShe83, FaSh83, KohayakawaRetterRodl, LM98, LuWang}.

For a graph $G$, we write $V(G)$ and $E(G)$ for its vertex and edge sets, $\Delta(G)$
for its maximum degree, and $G\cup H$ for the disjoint union of graphs $G$ and $H$. In
particular, $tK_2$ denotes a matching of size $t$ and $sK_1$ denotes the edgeless graph
on $s$ vertices. For a nonempty graph $G$, its \emph{isolate-free core} $\core(G)$ is
obtained by deleting all isolated vertices of $G$.

In the special case of a matching versus a fixed graph, Erd\H{o}s and Faudree proved
more generally that for fixed nontrivial graphs $H$ and $G$ the limit
\[
\lim_{t\to\infty}\frac{\sr(tH,G)}{t\,\sr(H,G)}
\]
exists \cite{ErdosFaudree84}. They then studied the special case $H=K_2$ and wrote
\[
\sri(G)=\lim_{t\to\infty}\frac{\sr(tK_2,G)}{t\,\sr(K_2,G)}.
\]
Since $\sr(K_2,G)=|E(G)|$, one may write
\[
\sri(G)=\lim_{t\to\infty}\frac{\sr(tK_2,G)}{t|E(G)|}.
\]
In our setting the existence of this limit will also follow from a direct subadditivity
argument together with Fekete's lemma.

Erd\H{o}s and Faudree  \cite{ErdosFaudree84} proved that $\sri(G)=1$ for stars, for graphs of the form
$K_{1,d}+e$, and for every connected graph on at most four vertices other than $P_4$. 
Javadi and Omidi 
\cite{JavadiOmidi2018} obtained the exact value of $\hat{R}(K_n, tK_2)$ for every pair of positive integers $n, t$, and as a byproduct, we give an affirmative answer to the question of Erd\H{o}s and Faudree.

Erd\H{o}s and Faudree asked the following two problems.
\begin{problem}[Erd\H{o}s--Faudree]\label{prob:EF1}
Is there an infinite family of graphs $\{G_N:N>1\}$ with
\[
|V(G_N)|=N,\qquad \Delta(G_N)<N-1,
\]
such that
$\sri(G_N)=1$
for every $N$?
\end{problem}

\begin{problem}[Erd\H{o}s--Faudree]\label{prob:EF2}
If $\{G_N:N>1\}$ is a family of graphs with $|V(G_N)|=N$ and
$\Delta(G_N)<c$
for some fixed constant $c$, must one have
\[
\lim_{N\to\infty}\sri(G_N)=0?
\]
\end{problem}

At first sight the two questions point in opposite directions. Problem~\ref{prob:EF1}
asks whether the limiting ratio can remain as large as possible, while
Problem~\ref{prob:EF2} asks whether bounded degree forces it to vanish. Yet the two
questions share the same hidden obstruction: isolated vertices are invisible to all matching
size Ramsey numbers. Indeed,
\[
G_N=K_2\cup (N-2)K_1
\]
already satisfies $|V(G_N)|=N$, $\Delta(G_N)=1<N-1$, and $\sri(G_N)=1$ for every
$N\ge 3$, so it solves Problem~\ref{prob:EF1} and simultaneously refutes
Problem~\ref{prob:EF2} for every fixed $c>1$.

The point of this paper is that this phenomenon is not accidental. It is a manifestation of a
general structural principle: for every nonempty graph $G$, the whole sequence
$\bigl(\sr(tK_2,G)\bigr)_{t\ge 1}$ depends only on the isolate-free core $\core(G)$,
obtained from $G$ by deleting all isolated vertices.
Once this is isolated, the two original Erd\H{o}s--Faudree problems are easy in their literal
form, and one can then ask what remains after the padding phenomenon has been removed.

For Problem~\ref{prob:EF1}, the natural successor is the connected theory. Since
$\sri(G)$ is itself a limiting parameter, the right connected question is not whether one can
force the exact identity $\sri(G_N)=1$ at every order, but whether one can keep
$\sri(G_N)$ asymptotically as large as possible, and more generally which limits can occur
along connected $N$-vertex families. We prove a complete answer: every value in $[0,1]$
arises, even inside the class of connected bipartite graphs.

For Problem~\ref{prob:EF2}, connectivity is not the only meaningful restriction. A more
precise formulation asks what happens under a fixed maximum-degree bound when the
isolate-free core itself is forced to grow. We prove that in this regime $\sri(G_N)$ always
tends to $0$; connected bounded-degree families appear as a special case because a
connected nonempty graph has no isolated vertices.

The three principal results may be summarized as follows.

\begin{theorem}[Summary theorem]\label{thm:summary}
The following statements hold.
\begin{enumerate}[label={\rm (\alph*)}]
\item For every nonempty graph $G$ and every integer $t\ge 1$,
$\sr(tK_2,G)=\sr(tK_2,\core(G))$.
Consequently, Problem~\ref{prob:EF1} has a positive answer and
Problem~\ref{prob:EF2} has a negative answer in their original form.

\item For every $\alpha\in[0,1]$ there exists $N_0$ and a family of connected bipartite
graphs $(G_N)_{N\ge N_0}$ such that $|V(G_N)|=N$ and
$\sri(G_N)\to\alpha$.
In particular, there exists a connected family with
\[
|V(G_N)|=N,\qquad \Delta(G_N)=N-2<N-1,\qquad \sri(G_N)\to 1.
\]

\item Let $c>1$ be fixed, and let $(G_N)$ be a family of nonempty graphs with
$|V(G_N)|=N$ and $\Delta(G_N)<c$. If
$|V(\core(G_N))|\to\infty,$
then
$\sri(G_N)\to 0$.
More quantitatively, then
\[
\sri(G_N)=O_c\!\left(\left(\frac{\log |V(\core(G_N))|}{|V(\core(G_N))|}\right)^{1/(c-1)}\right).
\]
In particular, every connected bounded-degree family satisfies $\sri(G_N)\to 0$.
\end{enumerate}
\end{theorem}

The rest of the paper is organized accordingly. Section~2 proves the padding principle and
settles the two original problems exactly as stated. Section~3 treats the connected theory
associated with Problem~\ref{prob:EF1}. Section~4 gives the strengthened positive form of
Problem~\ref{prob:EF2} in terms of growth of the isolate-free core.

\section{The Erd\H{o}s--Faudree problems and the isolate-free core}

We begin with a standard observation.

\begin{proposition}\label{prop:baseK2}
For every nonempty graph $G$,
\[
\sr(K_2,G)=|E(G)|.
\]
\end{proposition}

\begin{proof}
For any graph $F$ with $|E(G)|-1$ edges, we color all edges of $F$ blue, then there is neither a red copy
of $K_2$ nor a blue copy of $G$. Hence $F\not\to(K_2,G)$, so $\sr(K_2,G)\ge |E(G)|$.

Conversely, every red--blue coloring of $E(G)$ either contains a red edge,
and therefore a red copy of $K_2$, or else all edges are blue, and therefore form a blue
copy of $G$. Thus $G\to(K_2,G)$ and $\sr(K_2,G)\le |E(G)|$.
\end{proof}

\begin{proposition}\label{prop:subadd}
For every nonempty graph $G$ and all integers $s,t\ge 1$,
\[
\sr((s+t)K_2,G)\le \sr(sK_2,G)+\sr(tK_2,G).
\]
\end{proposition}

\begin{proof}
Choose graphs $F_s\to(sK_2,G)$ and $F_t\to(tK_2,G)$ with
$|E(F_s)|=\sr(sK_2,G)$ and $|E(F_t)|=\sr(tK_2,G)$. Let
$F=F_s\cup F_t$
be the disjoint union. Consider any red--blue coloring of $F$. If one of the two components
already contains a blue copy of $G$, and we are done. Otherwise $F_s$ contains a red copy
of $sK_2$ and $F_t$ contains a red copy of $tK_2$; since the components are disjoint, these
together form a red copy of $(s+t)K_2$. Thus $F\to((s+t)K_2,G)$, which proves the inequality.
\end{proof}

\begin{lemma}[Fekete's lemma \cite{Fekete23}]\label{lem:fekete}
Let $(a_t)_{t\ge 1}$ be a subadditive sequence of real numbers, that is,
$a_{s+t}\le a_s+a_t
~(s,t\ge 1)$.
Then
\[
\lim_{t\to\infty}\frac{a_t}{t}
=
\inf_{t\ge 1}\frac{a_t}{t}.
\]
\end{lemma}

\begin{corollary}[Fekete reduction]\label{cor:fekete-reduction}
For every nonempty graph $G$,
\[
\sri(G)=\inf_{t\ge 1}\frac{\sr(tK_2,G)}{t|E(G)|}.
\]
Moreover,
$0 \le \sri(G)\le 1$,
and
$\sri(G)=1$
if and only if 
$\sr(tK_2,G)=t|E(G)|$
for every $t\ge 1$.
In particular, for every integer $t\ge 1$,
\[
\sri(G)\le \frac{\sr(tK_2,G)}{t|E(G)|}.
\]
\end{corollary}

\begin{proof}
By Proposition~\ref{prop:subadd}, the sequence $a_t(G)=\sr(tK_2,G)$ is subadditive.
Hence Lemma~\ref{lem:fekete} yields
\[
\lim_{t\to\infty}\frac{a_t(G)}{t}=
\inf_{t\ge 1}\frac{a_t(G)}{t}.
\]
Using Proposition~\ref{prop:baseK2}, namely $\sr(K_2,G)=|E(G)|$, gives the displayed
formula for $\sri(G)$.

Finally, a graph with fewer than $t$ edges cannot contain a matching of size $t$.
Hence, if $|E(F)|<t$ and all edges of $F$ are colored red, then there is neither a red copy
of $tK_2$ nor a blue copy of $G$. Therefore, 
$\sr(tK_2,G)\ge t$.
On the other hand, the disjoint union of $t$ copies of $G$ shows that
$\sr(tK_2,G)\le t|E(G)|$.
Therefore every ratio $\sr(tK_2,G)/(t|E(G)|)$ lies in $[1/|E(G)|,1]$, and so does its
infimum. The infimum equals $1$ if and only if every term equals $1$.
\end{proof}

\begin{lemma}[Padding invariance]\label{lem:padding}
Let $H$ be a nonempty graph, let $s\ge 0$, and let $t\ge 1$. Then
\[
\sr(tK_2,H\cup sK_1)=\sr(tK_2,H).
\]
In particular,
$\sri(H\cup sK_1)=\sri(H)$.
\end{lemma}

\begin{proof}
The inequality
$\sr(tK_2,H)\le \sr(tK_2,H\cup sK_1)$
is immediate, since every blue copy of $H\cup sK_1$ contains a blue copy of $H$.

For the reverse inequality, choose a graph $F\to(tK_2,H)$ with
$|E(F)|=\sr(tK_2,H)$, and define
$F':=F\cup sK_1$.
Then $|E(F')|=|E(F)|$. In any red--blue coloring of $F'$, either the restriction to $F$
contains a red copy of $tK_2$, or else it contains a blue copy of $H$; in the latter case, the
$s$ isolated vertices of $F'$ enlarge this blue copy to a blue copy of $H\cup sK_1$. Thus
$F'\to(tK_2,H\cup sK_1)$, so
$\sr(tK_2,H\cup sK_1)\le |E(F')|=\sr(tK_2,H)$.
The statement for $\sri (H\cup sK_1)$ follows from Proposition~\ref{prop:baseK2}.
\end{proof}

For a nonempty graph $G$, let $\core(G)$ denote the graph obtained from $G$ by deleting
all isolated vertices.

\begin{corollary}[Core reduction]\label{cor:core}
For every nonempty graph $G$ and every integer $t\ge 1$,
\[
\sr(tK_2,G)=\sr(tK_2,\core(G))
\qquad\text{and}\qquad
\sri(G)=\sri(\core(G)).
\]
\end{corollary}

\begin{proof}
Write $G=\core(G)\cup sK_1$ for some $s\ge 0$ and apply
Lemma~\ref{lem:padding}.
\end{proof}

\begin{corollary}[Literal answers to Problems~\ref{prob:EF1} and \ref{prob:EF2}]
\label{cor:literal}
The family
\[
G_N=K_2\cup (N-2)K_1
\qquad (N\ge 3)
\]
satisfies
$|V(G_N)|=N$, $\Delta(G_N)=1<N-1$, and $\sri(G_N)=1$.
Consequently, Problem~\ref{prob:EF1} has a positive answer and
Problem~\ref{prob:EF2} has a negative answer for every fixed $c>1$.
\end{corollary}

\begin{proof}
Since $\sr(tK_2,K_2)=t$ for every $t\ge 1$, we have $\sri(K_2)=1$.
Lemma~\ref{lem:padding} gives
$\sri\bigl(K_2\cup (N-2)K_1\bigr)=\sri(K_2)=1$.
The degree statement is immediate.
\end{proof}

The previous corollary explains why the literal forms of the two problems are easy. The
next proposition records the same phenomenon in a slightly more structural way.

\begin{proposition}\label{prop:bounded-core}
Let $(G_N)$ be a family of nonempty graphs. If
\[
\sup_N |V(\core(G_N))|<\infty,
\]
then the set
$\{\sri(G_N):N\ge 1\}$
is finite.
\end{proposition}

\begin{proof}
By Corollary~\ref{cor:core}, the value $\sri(G_N)$ depends only on $\core(G_N)$. If all
cores have order at most $M$, then there are only finitely many graphs up to isomorphism on
at most $M$ vertices. Hence only finitely many values of $\sri(G_N)$ can occur.
\end{proof}

\begin{remark}
Proposition~\ref{prop:bounded-core} pinpoints the mechanism behind the literal negative
answer to Problem~\ref{prob:EF2}: one may keep the core bounded and pad by isolated
vertices forever. The natural question is therefore what happens when isolated vertices are
forbidden or, more generally, when the isolate-free core is forced to grow.
\end{remark}

\section{Connected families and Problem~\ref{prob:EF1}}

We now turn to the connected theory generated by Problem~\ref{prob:EF1}. Since
$\sri(G)$ is itself a limit parameter, the correct connected question is to determine which
limits are attainable along connected $N$-vertex families.

\begin{lemma}[Vertex deletion]\label{lem:delete-one}
Let $G$ be a nonempty graph and $s\ge 1$. If
$F\to(sK_2,G)$
and any $v\in V(F)$, then
$F-v\to((s-1)K_2,G)$.
\end{lemma}

\begin{proof}
Suppose that $F-v\not\to((s-1)K_2,G)$. Then there is a red--blue coloring of $F-v$ with
no red copy of $(s-1)K_2$ and no blue copy of $G$. Extend this coloring to $F$ by coloring
every edge incident to $v$ red. Any red matching in the extended coloring uses at most one
edge incident to $v$, so there is no red copy of $sK_2$. There is also still no blue copy of
$G$, because all newly colored edges are red. This contradicts $F\to(sK_2,G)$.
\end{proof}

\begin{proposition}[A degree lower bound]\label{prop:delta-lower}
For every nonempty graph $G$ and every integer $t\ge 1$,
\[
\sr(tK_2,G)\ge (t-1)\Delta(G)+|E(G)|.
\]
Consequently,
\[
\sri(G)\ge \frac{\Delta(G)}{|E(G)|}.
\]
\end{proposition}

\begin{proof}
Let $F\to(tK_2,G)$ and set $F_0=F$. For each $1\le i\le t-1$, choose a vertex $v_i$ of
maximum degree in $F_{i-1}$ and define
$F_i=F_{i-1}-v_i$.
By Lemma~\ref{lem:delete-one},
$F_i\to((t-i)K_2,G)
~ (0\le i\le t-1)$.
In particular, each $F_i$ contains a copy of $G$, because if all edges of $F_i$ are colored
blue then the arrowing property forces a blue copy of $G$. Therefore
$d_{F_{i-1}}(v_i)=\Delta(F_{i-1})\ge \Delta(G)
~ (1\le i\le t-1)$,
and also $|E(F_{t-1})|\ge |E(G)|$. Hence
\[
|E(F)|=
\sum_{i=1}^{t-1}d_{F_{i-1}}(v_i)+|E(F_{t-1})|
\ge (t-1)\Delta(G)+|E(G)|.
\]
Taking the minimum over all such $F$ proves the first inequality. Dividing by
$t|E(G)|$ yields
\[
\frac{\sr(tK_2,G)}{t|E(G)|}
\ge
\frac{(t-1)\Delta(G)+|E(G)|}{t|E(G)|}
\qquad (t\ge 1).
\]
Now take the infimum over $t$ and use Corollary~\ref{cor:fekete-reduction}.
\end{proof}

For integers $k\ge 2$ and
$\ell_1\ge \ell_2\ge \cdots\ge \ell_k\ge 0$,
let $B(\ell_1,\dots,\ell_k)$ be the graph obtained from $K_{k,k}$ by attaching $\ell_i$
pendant edges to the $i$-th vertex in one part of the bipartition. This graph is connected and
bipartite. If
\[
n=2k+\sum_{i=1}^k\ell_i,
\]
then
\begin{equation}\label{eq:basic-params}
|V(B)|=n,
\qquad
|E(B)|=k^2+\sum_{i=1}^k\ell_i=n+k(k-2),
\qquad
\Delta(B)=k+\ell_1.
\end{equation}

\begin{proposition}[A two-sided estimate for the bundled bipartite graphs]
\label{prop:bundle-bounds}
Let $B=B(\ell_1,\dots,\ell_k)$ with parameters as above. Then for every integer $t\ge 1$,
\[
(t-1)(k+\ell_1)+n+k(k-2)
\le
\sr(tK_2,B)
\le
n+(t-1)\ell_1+(k+t-1)^2-2k.
\]
\end{proposition}

\begin{proof}
The lower bound is Proposition~\ref{prop:delta-lower} together with
\eqref{eq:basic-params}.

For the upper bound, construct a bipartite graph $U_t$ as follows. Start with a complete
bipartite graph $K_{k+t-1,k+t-1}$ with left part $X$ and right part $Y$. Choose $k$ vertices
of $X$ and attach to them disjoint stars of sizes $\ell_1,\dots,\ell_k$, respectively. Attach to
each of the remaining $t-1$ vertices of $X$ a disjoint star of size $\ell_1$. Then
\[
|E(U_t)|=(k+t-1)^2+\sum_{i=1}^k\ell_i+(t-1)\ell_1
= n+(t-1)\ell_1+(k+t-1)^2-2k.
\]

We claim that for every set $S\subseteq V(U_t)$ with $|S|\le t-1$, the graph $U_t-S$
contains a copy of $B$. Call one of the $k+t-1$ attached stars \emph{damaged} if $S$
contains its center or one of its private leaves. Since each vertex of $S$ affects at most one
attached star, at most $t-1$ stars are damaged. Let $M$ be the multiset of star sizes
\[
M=\{\ell_1,\dots,\ell_k,\underbrace{\ell_1,\dots,\ell_1}_{t-1\text{ extra copies}}\}.
\]
After deleting at most $t-1$ entries from $M$, let
$\ell'_1\ge \ell'_2\ge \cdots\ge \ell'_k$
be the $k$ largest remaining entries. For each $i\in\{1,\dots,k\}$, the original multiset $M$ contains the elements $\ell_i$ for $1 \leq i \leq k$, together with $t-1$ additional copies of $\ell_1$. After removing at most $t-1$ entries, at least $k$ such entries
remain. Hence $\ell'_i\ge \ell_i$ for every $i$. Thus among the undamaged stars one can select
$k$ whose sizes, after reordering, dominate $(\ell_1,\dots,\ell_k)$ termwise. Moreover, at most
$t-1$ vertices of $Y$ are deleted, so at least $k$ vertices of $Y$ survive. Using those $k$
surviving vertices of $Y$ together with the chosen undamaged stars, one obtains a copy of
$B$ inside $U_t-S$. This proves the claim.

Now color $E(U_t)$ red and blue. If the red graph contains a matching of size $t$, then we
already have a red copy of $tK_2$. Otherwise, since $U_t$ is bipartite, K\H{o}nig's theorem
implies that the red graph has a vertex cover $S$ of size at most $t-1$. Then all edges of
$U_t-S$ are blue, and by the claim $U_t-S$ contains a blue copy of $B$. Therefore
$U_t\to(tK_2,B)$, which proves the upper bound.
\end{proof}

\begin{theorem}[Connected bipartite limit classification]\label{thm:connected-classification}
For every $\alpha\in[0,1]$ there exists $N_0$ and a family of connected bipartite graphs
$(G_N)_{N\ge N_0}$ such that $|V(G_N)|=N$ and
\[
\sri(G_N)\to \alpha.
\]
More precisely:
\begin{enumerate}[label={\rm (\roman*)}]
\item If $\alpha\in(0,1]$, one may choose a family with
\[
\begin{aligned}
|E(G_N)|&=N+O_{\alpha}(1),\\
\Delta(G_N)&=\alpha N+O_{\alpha}(1),\\
\sri(G_N)&=\alpha+O_{\alpha}(N^{-1/2}).
\end{aligned}
\]
\item If $\alpha=0$ and $t$ is sufficiently large, one may choose a family with
\[
\begin{aligned}
|E(G_N)|&\le 2N,\\
\Delta(G_N)&\le 2\sqrt N+2,\\
\sri(G_N)&=O(N^{-1/2}).
\end{aligned}
\]
\end{enumerate}
\end{theorem}

\begin{proof}
We treat the cases $\alpha>0$ and $\alpha=0$ separately.
\setcounter{case}{0}
\begin{case}
$\alpha\in(0,1]$.    
\end{case}

Choose an integer $k\ge 2$ such that $k\alpha \geq 1$. For each sufficiently large $N$, let
$s_N=N-2k$ and 
$\ell_1=\lfloor \alpha s_N\rfloor$.
Since $k\alpha \geq 1$, we have $s_N-\ell_1\le (k-1)\ell_1$ for all sufficiently large $N$. Hence
one can choose integers
\[
\ell_1\ge \ell_2\ge \cdots\ge \ell_k\ge 0
\qquad\text{with}\qquad
\sum_{i=1}^k\ell_i=s_N.
\]
Set
$G_N=B(\ell_1,\dots,\ell_k)$.
Then $G_N$ is connected and bipartite, $|V(G_N)|=N$, and by \eqref{eq:basic-params},
\[
|E(G_N)|=N+k(k-2)=N+O_{\alpha}(1),
\qquad
\Delta(G_N)=k+\ell_1=\alpha N+O_{\alpha}(1).
\]

By Proposition~\ref{prop:delta-lower}, we have 
\[
\sri(G_N)\ge \frac{k+\ell_1}{N+k(k-2)}=\alpha+O_{\alpha}(N^{-1}).
\]
On the other hand, Corollary~\ref{cor:fekete-reduction} and
Proposition~\ref{prop:bundle-bounds} give, for every integer $t\ge 1$,
\[
\sri(G_N)
\le
\frac{N+(t-1)\ell_1+(k+t-1)^2-2k}{t\bigl(N+k(k-2)\bigr)}.
\]
Since $k$ is fixed and $\ell_1=\alpha N+O_{\alpha}(1)$, the right-hand side is
\[
\alpha+O_{\alpha}\!\left(\frac{1}{t}+\frac{t}{N}\right),
\]
uniformly for all sufficiently large $N$ and all integers $t\ge 1$. Taking the infimum over
all integers $t\ge 1$ and using
\[
\inf_{t\ge 1}\left(\frac{1}{t}+\frac{t}{N}\right)=O(N^{-1/2}),
\]
we obtain
$\sri(G_N)\le \alpha+O_{\alpha}(N^{-1/2})$.
Combining the lower and upper bounds gives $\sri(G_N)\to\alpha$.

\begin{case}
$\alpha=0$.   
\end{case}

For each sufficiently large $N$, let
$q=\lfloor \sqrt N\rfloor$,
$s_N=N-2q$,
and write
$s_N=aq+r
~(0\le r<q)$.
Define
$\ell_1=\cdots=\ell_r=a+1$,
$\ell_{r+1}=\cdots=\ell_q=a$,
and let
$G_N:=B(\ell_1,\dots,\ell_q)$.
Then $G_N$ is connected and bipartite and has $|V(G_N)|=N$. Moreover,
\[
\ell_1\le \left\lceil\frac{N-2q}{q}\right\rceil +1\le \sqrt N+2,
\]
so by \eqref{eq:basic-params},
$|E(G_N)|=N+q(q-2)\le 2N$ and 
$\Delta(G_N)=q+\ell_1\le 2\sqrt N+2$.
Again Corollary~\ref{cor:fekete-reduction} and
Proposition~\ref{prop:bundle-bounds} imply that for every integer $t\ge 1$,
\[
\sri(G_N)
\le
\frac{N+(t-1)\ell_1+(q+t-1)^2-2q}{t\bigl(N+q(q-2)\bigr)}.
\]
Since $q=\lfloor\sqrt N\rfloor$ and $\ell_1\le \sqrt N+2$, the right-hand side is
\[
O\!\left(\frac{1}{2t}+\frac{1}{2\sqrt N}+\frac{t}{2(N-\sqrt N)}\right),
\]
uniformly for all sufficiently large $N$ and $t$.
Hence $\sri(G_N)\to 0$.
\end{proof}

\begin{remark}\label{rem:tN-choice}
In the proof of Theorem~\ref{thm:connected-classification}, the estimates for $\sri(G_N)$
are obtained directly from Corollary~\ref{cor:fekete-reduction} by taking the infimum over
all integers $t\ge 1$. No auxiliary choice of $t$ depending on $N$ is needed.
\end{remark}

\begin{corollary}[Asymptotic connected answer to Problem~\ref{prob:EF1}]
\label{cor:connected-one}
There exists a family of connected bipartite graphs $(G_N)_{N\ge 4}$ with
\[
|V(G_N)|=N,
\qquad
\Delta(G_N) \leq N-2<N-1,
\qquad
\sri(G_N)\to 1.
\]
\end{corollary}

\begin{proof}
Apply Theorem~\ref{thm:connected-classification} with $\alpha=1$. In the construction of
Case~1 one may take $k=2$, and then
$|E(G_N)|=N$ and 
$\Delta(G_N) \leq N-2.$
\end{proof}

\begin{remark}
Theorem~\ref{thm:connected-classification} gives a complete answer to the connected theory
naturally generated by Problem~\ref{prob:EF1}: every value in $[0,1]$ occurs as a limit of
$\sri(G_N)$ along connected $N$-vertex graphs, and this remains true inside the bipartite
category.
\end{remark}

\section{Problem~\ref{prob:EF2} beyond padding}

We now turn to the positive content behind Problem~\ref{prob:EF2}. The right condition is
not connectivity alone, but growth of the isolate-free core under a fixed maximum-degree
bound.

\begin{proposition}[Self-Ramsey reduction]\label{prop:self-ramsey}
Let $\nu(G)$ denote the maximum matching number of $G$. Then for every nonempty graph $G$,
\[
\sr(\nu(G)K_2,G)\le \sr(G,G),
\]
and consequently
\[
\sri(G)\le \frac{\sr(G,G)}{\nu(G)|E(G)|}.
\]
\end{proposition}

\begin{proof}
Let $F\to(G,G)$ with $|E(F)|=\sr(G,G)$. In any red--blue coloring of $F$, either there is a
blue copy of $G$, or there is a red copy of $G$. In the latter case, that red copy contains a
matching of size $\nu(G)$, hence a red copy of $\nu(G)K_2$. Therefore
$F\to(\nu(G)K_2,G)$,
which proves the first inequality. The second follows from
Corollary~\ref{cor:fekete-reduction},
\[
\sri(G)
\le \frac{\sr(\nu(G)K_2,G)}{\nu(G)|E(G)|}
\le \frac{\sr(G,G)}{\nu(G)|E(G)|}.
\]
\end{proof}

We shall use the following consequence of a theorem of
Kohayakawa, R\"odl, Schacht, and Szemer\'edi \cite[Theorem~1]{KRSS11}.

\begin{proposition}[A corollary of Kohayakawa--R\"odl--Schacht--Szemer\'edi]
\label{prop:KRSS}
For every fixed integer $\Delta\ge 2$ there exists a constant $C_{\Delta}>0$ such that every
$n$-vertex graph $H$ with $\Delta(H)\le \Delta$ satisfies
\[
\sr(H,H)\le C_{\Delta}\,n^{2-1/\Delta}(\log n)^{1/\Delta}.
\]
\end{proposition}

\begin{theorem}[Core-growth form of Problem~\ref{prob:EF2}]
\label{thm:core-growth}
Let $c>1$ be fixed, and let $(G_N)$ be a family of nonempty graphs with
$|V(G_N)|=N$ and 
$\Delta(G_N)<c$.
Set
\[
M_N=|V(\core(G_N))|.
\]
If $M_N\to\infty$, then
$\sri(G_N)\to 0$.
More quantitatively, 
then there exists a constant $K_c>0$ such that
\[
\sri(G_N)\le K_c\left(\frac{\log M_N}{M_N}\right)^{1/c-1}
\qquad (N\ge 1).
\]
\end{theorem}

\begin{proof}
Let
$H_N=\core(G_N)$.
By Corollary~\ref{cor:core},
$\sri(G_N)=\sri(H_N)$.
Since $\Delta(G_N)<c$ and degrees are integers, we have
$\Delta(H_N)\le c-1.$
Moreover, $H_N$ has no isolated vertices. Hence every vertex of $H_N$ has degree at least
$1$, and therefore
\[
|E(H_N)|\ge \frac{M_N}{2}.
\]
By Vizing's theorem, the edge set of $H_N$ can be partitioned into at most $\Delta(H_N)+1$
matchings, so
\[
\nu(H_N)\ge \frac{|E(H_N)|}{\Delta(H_N)+1}\ge \frac{|E(H_N)|}{c}.
\]
Applying Proposition~\ref{prop:self-ramsey} and then Proposition~\ref{prop:KRSS}, we
obtain
\[
\begin{aligned}
\sri(G_N)=\sri(H_N)
&\le \frac{\sr(H_N,H_N)}{\nu(H_N)|E(H_N)|}\\
&\le \frac{c\sr(H_N,H_N)}{|E(H_N)|^2}\\
&\le \frac{c K_{c}M_N^{2-1/(c-1)}(\log M_N)^{1/(c-1)}}
{|E(H_N)|^2}.
\end{aligned}
\]
Using $|E(H_N)|\ge M_N/2$, we find
\[
\sri(G_N)
\le 4c K_{c}
\left(\frac{\log M_N}{M_N}\right)^{1/(c-1)}.
\]
This proves the quantitative estimate after renaming the constant, and the estimate tends to
$0$ because $M_N\to\infty$.
\end{proof}

\begin{corollary}[Connected bounded-degree families]\label{cor:connected-two}
Let $c>1$ be fixed, and let $(G_N)$ be a family of connected graphs with
$|V(G_N)|=N$ and 
$\Delta(G_N)<c$.
Then 
$\sri(G_N)\to 0$.
More quantitatively, one has
\[
\sri(G_N)=O_c\!\left(\left(\frac{\log N}{N}\right)^{1/(c-1)}\right).
\]
\end{corollary}

\begin{proof}
Every connected nonempty graph has no isolated vertices, so $\core(G_N)=G_N$ and hence
$M_N=N$. The conclusion follows immediately from Theorem~\ref{thm:core-growth}.
\end{proof}

\begin{corollary}[Final resolution of the two Erd\H{o}s--Faudree problems]
\label{cor:final}
The two problems split sharply according to whether one allows padding by isolated vertices.
\begin{enumerate}[label={\rm (\roman*)}]
\item In their original literal form, Problem~\ref{prob:EF1} has a positive answer and
Problem~\ref{prob:EF2} has a negative answer.
\item In the connected theory generated by Problem~\ref{prob:EF1}, every value in $[0,1]$
occurs as a limit of $\sri(G_N)$ along connected bipartite $N$-vertex families; in particular,
one may have $\Delta(G_N)=N-2$ and still obtain $\sri(G_N)\to 1$.
\item In the strengthened form of Problem~\ref{prob:EF2}, fixed maximum degree together
with growth of the isolate-free core forces $\sri(G_N)\to 0$. In particular, every connected
bounded-degree family satisfies $\sri(G_N)\to 0$.
\end{enumerate}
\end{corollary}

\begin{proof}
Part~\textup{(i)} is Corollary~\ref{cor:literal}. Part~\textup{(ii)} follows from
Theorem~\ref{thm:connected-classification} and Corollary~\ref{cor:connected-one}.
Part~\textup{(iii)} is exactly Theorem~\ref{thm:core-growth} and
Corollary~\ref{cor:connected-two}.
\end{proof}

The paper therefore leaves a very sharp picture. The original Erd\H{o}s--Faudree questions
are easy only because isolated vertices are invisible to the matching size Ramsey numbers. The
first problem becomes rich and flexible in the connected setting, where every limit in $[0,1]$
can occur. The second problem becomes true as soon as the bounded-degree isolate-free core
is forced to grow, and connectivity is one particularly natural way to guarantee that growth.

\end{document}